\documentclass[a4paper, 10pt,fleqn]{amsart}
\usepackage{amsmath}
\usepackage{amssymb, color, hyperref}

\theoremstyle{plain}
\newtheorem{theorem}{\bf Theorem}[section]
\newtheorem{proposition}[theorem]{\bf Proposition}
\newtheorem{lemma}[theorem]{\bf Lemma}
\newtheorem{corollary}[theorem]{\bf Corollary}

\theoremstyle{definition}
\newtheorem{example}[theorem]{\bf Example}

\newtheorem{definition}[theorem]{\bf Definition}
\newtheorem{remark}[theorem]{\bf Remark}

\newcommand{\N}{\mathbb N}
\newcommand{\Z}{\mathbb Z}
\newcommand{\R}{\mathbb R}
\newcommand{\Q}{\mathbb Q}

\newcommand{\DP}{\negthinspace : \negthinspace}

\newcommand{\BF}{\text{\rm BF}}

\DeclareMathOperator{\spec}{spec}

\newcommand{\red}{{\text{\rm red}}}

\numberwithin{equation}{section}

\begin{document}

\title{On strongly primary monoids and domains}

\address{Institute for Mathematics and Scientific Computing\\ University of Graz, NAWI Graz\\ Heinrichstra{\ss}e 36\\ 8010 Graz, Austria }
\email{alfred.geroldinger@uni-graz.at}
\urladdr{https://imsc.uni-graz.at/geroldinger}

\address{Department of Mathematics \\ University of Haifa \\ 199 Abba Khoushy Avenue, Mount Carmel, Haifa 3498838, Israel}
\email{mroitman@math.haifa.ac.il}

\author{Alfred Geroldinger and Moshe Roitman}

\thanks{This work was supported by the Austrian Science Fund FWF, Project Number P28864-N35}



\begin{abstract}
	A commutative integral domain is primary if and only if it is one-dimensional and local. A   domain  is strongly primary if and only if it is local and each nonzero  principal ideal contains a power of the maximal ideal. Hence one-dimensional local Mori domains are strongly primary. We prove among other results, that if $R$ is  a domain such that the conductor $(R:\widehat R)$ vanishes, then $\Lambda(R)$ is finite, that is, there exists a positive integer $k$ such that each non-zero non-unit of $R$ is a product of at most $k$ irreducible elements. Using this result we obtain that every strongly primary domain is locally tame, and that a domain $R$ is globally tame if and only if $\Lambda(R)=\infty$. In particular, we answer Problem 38 in \cite{C-F-F-G14} in the affirmative. Many of our results are formulated for  monoids.
\end{abstract}

\maketitle

\section{Introduction} \label{1}

Factorization theory of integral domains studies  factorizations of elements and ideals (\cite{An97, Ge-HK06a, Fo-Ho-Lu13a}). Some ideal theoretic conditions (such as the ascending chain condition on principal ideals) guarantee that every non-zero non-unit element of a domain can be written as a product of atoms (irreducible elements).
The goal is to describe the non-uniqueness of factorizations by arithmetical invariants and study their relationship with classical ring theoretical invariants. Class groups and the structure of certain localizations are such algebraic invariants which control element factorizations. In case of  weakly Krull Mori domains this interplay is described by the  $T$-block monoid of the domain which is built by the $v$-class group and  localizations at height-one primes containing the conductor  (\cite[Theorem 3.7.1]{Ge-HK06a}). These localizations are one-dimensional and local, and this connection stimulated the interest of factorization theory in one-dimensional local domains.

To recall some arithmetical invariants, let $R$ be an integral domain. If $a = u_1 \cdot \ldots \cdot u_k$ is a factorization of an element $a \in R$ into atoms  $u_1, \ldots, u_k$, then $k$ is a factorization length. The set $\mathsf L (a) \subseteq \N$ of all possible factorization lengths of $a$, is called the {\it set of lengths} of $a$. The 	{\it local tame degree} $\mathsf t (u)$ of an atom $u \in R$ is the smallest $N$ with the following property: for any multiple $a$ of $u$ and any factorization $a = v_1 \cdot \ldots \cdot v_n$ of $a$ there is a  subproduct which is a multiple of $u$, say $v_1 \cdot \ldots \cdot v_m$, and a refactorization of this subproduct which contains $u$, say $v_1 \cdot \ldots \cdot v_m = u u_2 \cdot \ldots \cdot u_{\ell}$ such that $\max \{\ell, m \} \le N$. Local tameness is a central finiteness property which -- in many settings -- implies further arithmetical finiteness properties (see, e.g., \cite{Ge-HK06a}, \cite[Section 4]{Ge-Go-Tr20}).

Here are two classes of locally tame monoids:
 \begin{itemize}
	\item Krull monoids with finite class group \cite[Theorem 3.4.10]{Ge-HK06a}.
	
	\item C-monoids  \cite[Theorem 3.3.4]{Ge-HK06a}.
	\end{itemize}

Moreover,  if $R$ is a Mori domain  with non-zero  conductor  $\mathfrak f=(R\DP \widehat R)$ such that $R/\mathfrak f$ and the class group $\mathcal C ( \widehat R)$ are finite,  then   the multiplicative monoid $R^{\bullet}=R\setminus \{0\}$ is a  C-monoid by \cite[Theorem 2.11.9]{Ge-HK06a}, whence $R$ is locally tame (compare with Example \ref{3.13}.4).
Precise values for local and global tame degrees were studied for Krull monoids  and numerical monoids (\cite{C-G-L09,Om12a, C-C-M-M-P14,Ga-Ge-Sc15a}). For computational aspects  and for results in the non-commutative setting we refer to (\cite{Ba-Sm15, GS16a}).

By a {\it local} ring, we mean a commutative ring with a unique maximal ideal, not necessarily noetherian.
It is well-known that a domain $R$ is one-dimensional  and local if and only if its multiplicative monoid of non-zero elements is primary. A monoid $H$ is {\it strongly primary} if each principal ideal of $H$ contains a power of the maximal ideal.
The multiplicative monoid of a one-dimensional local Mori domain is strongly primary, and this was the key property to prove local tameness for one-dimensional local Mori domains under a variety of additional assumptions (\cite[Theorem 3.5]{Ge-Ha-Le07} and \cite[Theorem 3.5]{Ha02a}). However, the general case remained open.

In the present paper we prove that every strongly primary domain is  locally tame  and   provide a characterization of global tameness  (Theorems \ref{main} (b) and \ref{tame}.2). In particular, all one-dimensional local Mori domains turn out to be locally tame and this answers Problem 38 in \cite{C-F-F-G14} in the affirmative.
Although our present approach is semigroup theoretical over large parts (Theorem \ref{main} (a)), it also uses substantially  the ring structure, and this is unavoidable since strongly primary Mori monoids need not be locally tame as shown in  \cite[Proposition 3.7 and Example 3.8]{Ge-Ha-Le07} (see Example \ref{3.13}.1).

\section{Background on primary monoids and domains} \label{2}

We denote by $\N$ the set of positive integers and by $\N_0 = \N \cup \{0\}$ the set of non-negative integers.
For integers $a, b \in \Z$, we denote by $[a,b] = \{ x \in \Z \mid a \le x \le b\}$ the discrete interval between $a$ and $b$. Let $L, L' \subseteq \Z$ be subsets of the integers. Then $L+L' = \{a+b \mid a \in L, b \in L'\}$ denotes the sumset. A positive integer $d \in \N$ is a distance of $L$ if there is some $a \in L$ such that $L \cap [a, a+d] = \{a, a+d\}$. We  denote by $\Delta (L) \subseteq \N$ the set of distances of $L$. Thus $\Delta (L)=\emptyset$ if and only if $|L| \le 1$, and  $L$ is an arithmetical progression with difference $d$ if and only if $\Delta (L) \subseteq \{d\}$. Note  that the empty set is considered an arithmetical progression.

By a {\it monoid}, we mean a commutative cancellative semigroup with identity and by a {\it domain}, we mean a commutative integral domain.  Thus, if $R$ is a domain, then $R^{\bullet} = R \setminus \{0\}$ is a monoid. All ideal theoretic and arithmetical concepts are introduced in a monoid setting, but they will be  used both for monoids and  domains.

\smallskip
\noindent
{\bf Ideal theory of monoids and domains.} Let $H$ be a monoid and $\mathsf q (H)$ the quotient group of $H$. We denote by
\begin{itemize}
\item $H' = \{ x \in \mathsf q (H) \mid \text{there exists some $N \in \N$ such that $x^n \in H$ for all $n \ge N$} \}$ the {\it seminormal closure} of $H$, by

\item $\widetilde H = \{ x \in \mathsf q (H) \mid \text{there exists some $N \in \N$ such that $x^N \in H$} \}$ the {\it root closure} of $H$, and by

\item $\widehat H = \{ x \in \mathsf q (H) \mid \text{there exists $c \in H$ such that $cx^n \in H$ for all $n \in \N$} \}$ the {\it complete integral closure} of $H$.
\end{itemize}
Then we have
\begin{equation} \label{eq2.1}
H \subseteq H' \subseteq \widetilde H \subseteq \widehat H \subseteq \mathsf q (H) \,,
\end{equation}
and  all inclusions can be strict.
We say that $H$ is {\it seminormal} ({\it root-closed}, resp. {\it completely integrally closed}) if $H=H'$ ($H = \widetilde H$, resp. $H = \widehat H$). Note that $H'$ is seminormal, $\widetilde H$ is root-closed, but $\widehat H$ need not be completely integrally closed.

\smallskip
\begin{lemma} \label{2.1}
Let $H$ be a root-closed monoid and $x \in \mathsf q (H)$.
\begin{enumerate}
\item If $c \in H$ such that $cx^n \in H$ for some $n \in \N$, then $cx^k \in H$ for every $k \in [1,n]$.

\item We have $x \in \widehat H$ if and only if there exists $c \in H$ such that $cx^n \in H$ for infinitely many $n \in \N$.
\end{enumerate}
\end{lemma}

\begin{proof}
1. Let $c \in H$ and $n \in \N$ such that $cx^n \in H$. If $k \in [1,n]$, then  $(cx^k)^n = (c x^n)^k c^{n-k} \in H$ whence $cx^k \in H$ because $H$ is root-closed.

2. If there is a $c \in H$ such that $c x^n \in H$ for infinitely many $n \in \N$, then 1. implies that $c x^k \in H$ for all $k \in \N$ whence $x \in \widehat H$.
\end{proof}

\begin{lemma} \label{2.2}
Let $H$ be a monoid.
\begin{enumerate}
\item $\widehat H$ is root-closed.

\item If $(H \colon \widehat H) \ne \emptyset$, then $\widehat H$ is completely integrally closed.

\item An element $x \in \mathsf q (H)$ lies in $\widehat{\widehat H}$ if and only if there exists an element $c \in H$ such that $c x^n \in \widehat H$ for infinitely many $n \in \N$.
\end{enumerate}
\end{lemma}

\begin{proof}
1. Let $x \in \mathsf q (H)$ such that $x^e \in \widehat H$ for some $e \in \N$. We have to show that $x \in \widehat H$. There is an element $c \in H$ such that $c x^i \in H$ for every $i \in [1,e]$ and such that $c(x^e)^k \in H$ for all $k \in \N_0$.  If $n \in \N$, then $n = k e + i$, with $k \in \N_0$ and $i \in [1,e]$, and $c^2 x^n = c(x^e)^k (c x^i) \in H$ which implies that $x \in \widehat H$.

2. See \cite[Propositon 2.3.4]{Ge-HK06a}.

3. Since $\widehat H$ is root-closed by 1., the assertion follows from Lemma \ref{2.1}.2 (applied to the monoid $\widehat H$).
\end{proof}

Let $H$ be a monoid. Then $H^{\times}$ denotes the group of units of $H$ and $H_{\red} = \{a H^{\times} \mid a \in H \}$ its associated reduced monoid. We also let $\mathfrak m=H\setminus H^{\times}$ and $\mathfrak n=\widehat H\setminus \widehat H^{\times}$.
Let $X, Y  \subseteq \mathsf q (H)$ be  subsets. Then $X$ is an $s$-ideal of $H$ if $X \subseteq H$ and $XH =X$. We set $(X \colon Y)= \{z \in \mathsf q (H) \mid zY \subseteq X \}$ and $X^{-1} = (H \DP X)$. A {\em divisorial ideal ($v$-ideal)} is a set of the form $(X^{-1})^{-1}$ for $X\subseteq \mathsf q(H)$. We denote by $v$-$\spec (H)$ the set of all prime $v$-ideals of $H$. The monoid $H$ is a {\it Mori monoid}  if it satisfies the ascending chain condition on divisorial ideals, and it is a {\it Krull monoid} if it is a completely integrally closed Mori monoid.

\smallskip
\noindent
{\bf Arithmetic of monoids and domains.} For any set $P$, let $\mathcal F (P)$ be the free abelian monoid with basis $P$. Let $|\cdot | \colon \mathcal F (P) \to \N_0$ denote the unique epimorphism satisfying $|p|=1$ for each $p \in P$,  whence $|\cdot|$ is mapping each $z \in \mathcal F (P)$ onto its length. We denote by $\mathcal A (H)$ the set of atoms of $H$. Then $\mathsf Z (H) = \mathcal F ( \mathcal A (H_{\red}))$ is the {\it factorization monoid} of $H$ and $\pi \colon \mathsf Z (H) \to H_{\red}$ denotes the canonical epimorphism. For an element $a \in H$,
\begin{itemize}
\item $\mathsf Z (a) = \pi^{-1} (aH^{\times}) \subseteq \mathsf Z (H)$ is the {\it set of factorizations } of $a$, and

\item $\mathsf L (a) = \{ |z| \mid z \in \mathsf Z (a) \} \subseteq \N_0$ is the {\it set of lengths} of $a$.
\end{itemize}
To define the distance of factorizations, consider two factorizations $z, z' \in \mathsf Z (H)$. Then we write
\[
z= u_1 \cdot \ldots \cdot u_{\ell}v_1 \cdot \ldots \cdot v_m \quad \text{and} \quad z'= u_1 \cdot \ldots \cdot u_{\ell}w_1 \cdot \ldots \cdot w_n \,,
\]
where $\ell, m,n \in \N_0$ and all $u_i,v_j,w_k \in \mathcal A (H_{\red})$ such that $\{v_1, \ldots, v_m\} \cap \{w_1, \ldots, w_n\} = \emptyset$. We call $\mathsf d (z,z') = \max \{m,n\} \in \N_0$ the {\it distance} between $z$ and $z'$. The function $\mathsf d \colon \mathsf Z (H) \times \mathsf Z (H)\to \N_0$ is a metric on $\mathsf Z (H)$.
We say that $H$ is
\begin{itemize}
\item {\it atomic} if every non-unit can be written as a finite product of atoms, and
\item a BF-{\it monoid} (a {\it bounded factorization monoid}) if it is atomic and all sets of lengths are finite.
\end{itemize}
A monoid is a BF-monoid if and only if $\bigcap_{n \ge 0} (H \setminus H^{\times})^n = \emptyset$, and
 Mori monoids are BF-monoids (\cite[Theorem 2.2.9]{Ge-HK06a}). For every $k \in \N$, we set $\rho_k (H)=k$ if $H=H^{\times}$, and otherwise we set
\[
\rho_k (H) = \sup \{\sup \mathsf L (a) \mid a \in H, k \in \mathsf L (a) \} \in \N \cup \{\infty\} \,.
\]
Clearly, the sequence $(\rho_k (H))_{k \ge 1}$ is increasing and, if $\rho_k (H)$ is finite for some $k \in \N$, then $\rho_k (H)$ is the maximal length of a factorization of a product of $k$ atoms. The {\it elasticity} of $H$, introduced by Valenza in \cite{Va90}, is defined as
$\rho (H) = \sup \{ m/n \mid m,n \in L, L \in \mathcal L (H) \}$, and by \cite[Proposition 1.4.2]{Ge-HK06a} we have

\begin{equation} \label{elasticities}
\rho (H) = \sup \{ m/n \mid m,n \in L, L \in \mathcal L (H) \} = \lim_{k \to \infty} \frac{\rho_k (H)}{k}.
\end{equation}
For a subset $S \subseteq H$, the {\it set of distances} of $S$ is defined by
\[
\Delta (S) = \bigcup_{a \in S} \Delta \big( \mathsf L (a) \big) \quad \text{and we set} \quad \Lambda (S) = \sup \{ \min \mathsf L (a) \mid a \in S \} \in \N_0 \cup \{\infty\} \,.
\]
For $a \in H$, let $\Lambda (a) = \Lambda ( \{a\})$. If $a, b \in H$, then $\mathsf L (a) + \mathsf L (b) \subseteq \mathsf L (ab)$ whence $\min \mathsf L (ab) \le \min \mathsf L (a) + \min \mathsf L (b)$ and thus
\begin{equation} \label{Lambda-inequality}
\Lambda (ab) \le \Lambda (a) + \Lambda (b) \,.
\end{equation}

\smallskip

The {\it catenary degree} $\mathsf c (a)$ of an element $a \in H$ is the smallest $N \in \N_0 \cup \{\infty\}$ with the following property: if $z, z' \in \mathsf Z (a)$ are two factorizations of $a$, then there are factorizations $z=z_0, z_1, \ldots, z_k=z'$ of $a$ such that $\mathsf d (z_{i-1}, z_i) \le N$ for all $i \in [1,k]$.
Then
\[
\mathsf c (H) = \sup \{\mathsf c (a) \mid a \in H \} \in \N_0 \cup \{\infty\}
\]
is the {\it catenary degree} of $H$. The monoid $H$ is factorial if and only if $\mathsf c (H)=0$,  and if this does not hold, then
\begin{equation} \label{catenary}
2 + \sup \Delta (H) \le \mathsf c (H)  \quad \text{by} \quad \text{\cite[Theorem 1.6.3]{Ge-HK06a}} \,.
\end{equation}

\smallskip
\noindent
{\bf Primary monoids and domains.}
Let $H$ be a monoid  and $\mathfrak m = H \setminus H^{\times}$. Then $H$ is said to be
\begin{itemize}
\item {\it primary} if  $H \ne H^{\times}$ and for all $a, b \in \mathfrak m$ there is an $n \in \N$ such that $b^n \in aH$, and

\item {\it strongly primary} if $H \ne H^{\times}$ and for every $a \in \mathfrak m$ there is an $n \in \N$ such that $\mathfrak m^n \subseteq aH$ (we denote by $\mathcal M (a)$ the smallest $n \in \N$ having this property).
\end{itemize}
If $H$ is  strongly primary, then $H$ is a primary BF-monoid  (\cite[Lemma 2.7.7]{Ge-HK06a}).  However,  primary BF-monoids need not be strongly primary (Example \ref{3.13}.2).

\begin{lemma}\cite[Proposition 1]{Ge96} \label{2.3}
	Let $H$ be a primary monoid.
Then $H^{\times} = {H'}^{\times} \cap H = \widetilde H^{\times} \cap H = \widehat H^{\times} \cap H$.
	\end{lemma}

\begin{lemma}\label{prim}
Let $H$ be a primary monoid.
	\begin{enumerate}
		\item
		If   $a\in H$ and $x\in \mathsf q(H)$, then there exists an $N\in \mathbb N$
		such that $a^Nx\in H$, so  $a^nx\in H$ for all $n\ge N$. Moreover, for $n$ sufficiently large, we have both $a^nx\in H$ and $a^n\in Hx$.
		
		\item
		If $H$ is  strongly primary and $x\in \mathsf q(H)$, then, for $N\in \mathbb N$ sufficiently large, $\mathfrak m^Nx\subseteq  H$ and $\mathfrak m^N\subseteq  Hx$.
	\end{enumerate} 	
\end{lemma}

\begin{proof}
1. Let $x= bc^{-1}$, where $b,c\in H$. We have $Hb\subseteq Hbc^{-1}=Hx$. Since $H$ is primary, $Hb$ and so also $Hx$ contains a power of $a$.  Hence also $Hx^{-1}$ contains a power of $a$.  Thus for $n$ sufficiently large we have both $a^nx\in H$ and $a^n\in Hx$.
		
2. Use a similar argument as in  the previous item.
\end{proof} 	  	

If $H$ is strongly primary and  $x\in \mathsf q(H)$, we denote by $\mathcal M (x) \in \N$ the smallest $n \in \N$ with $\mathfrak m^n \subseteq xH$ (see Lemma \ref{prim}.2).

\begin{lemma} \label{2.4}
	Let $H$ be a primary monoid.
	\begin{enumerate}
		\item $H' = \widetilde H$. In particular, $H$ is seminormal if and only if $H$ is root-closed.
		
		\item
		 If $H$ is seminormal, then $ \mathfrak m \subseteq (H \colon \widehat H)$, so $\widehat H = (\mathfrak m \colon \mathfrak m)$.
		
		\item \cite[Theorem 4]{Ge93b}	
		$\widehat{\widetilde H}$ is completely integrally closed. Thus the complete integral closure of a seminormal primary monoid is completely integrally closed.
	\end{enumerate}
\end{lemma}

\begin{proof}
	1. By  \eqref{2.1}, we have $H' \subseteq \widetilde H$ whence it remains to verify the reverse inclusion. Let $x \in \mathsf q (H)$,  and $k \in \N$ such that $x^k \in H$. By Lemma \ref{prim}.1, there is an $N \in \mathbb N$ such that  $(x^k)^nx \in H$ for all $n \ge N$. Thus $x^{kn} \in H$ and $x^{kn+1} \in H$ for all $n \ge N$, which implies that $x^m \in H$ for all $m$ sufficiently large since $kn$ and $kn+1$ are coprime integers (explicitly, $x^m \in H$ for $m \ge (kn)^2$). Therefore  $x \in H'$.
	
2. Let $x\in \widehat H$, thus $dx^n\in H$ for some $d\in H$ and all $n\in \mathbb N$. Let $a\in H$. Since $H$ is primary, we have $a^kd^{-1}\in H$ for some $k \in\mathbb N$, so $a^kx^n\in H$ for all $n\in \mathbb N$. Thus $a^k(x^n)^k \in H$ for all $n\in \mathbb N$. Since $H$ is root-closed by item 1., we obtain that $ax^n\in H$ for all $n\in\mathbb N$. Thus $a\widehat H\subseteq H$, so  $\mathfrak m\widehat H \subseteq H$. Since $\mathfrak m \subseteq  \mathfrak n$ by Lemma \ref{2.3}, we infer that $\mathfrak m\widehat H \subseteq\mathfrak m$, that is, $\widehat H\subseteq (\mathfrak m\DP \mathfrak m)$, so $\widehat H=(\mathfrak m\DP \mathfrak m)$. Cf. \cite[Proposition 4.8]{G-HK-H-K03}.
\end{proof}

Monoid properties do not always carry over  to integral domains. However,  the domain $R$ is  seminormal (completely integrally closed, Mori, Krull, primary, strongly primary, atomic) if and only if its monoid $R^{\bullet}$ has the respective property. We consider,  for example, the Mori property.  By definition, the domain $R$ is Mori if and only if it satisfies the ascending chain condition on integral divisorial ideals. If $X \subseteq R$ is a subset, then  $I=(R:(R:X))$ is divisorial, and we have $I^{\bullet}=(R^{\bullet}:(R^{\bullet}:X^{\bullet}))$, and $I=(R^{\bullet}:(R^{\bullet}:X^{\bullet}))\cup \{0\}$, where $Y^{\bullet} = Y \setminus \{0\}$ for $Y \subseteq \mathsf q (R)$.  It follows that the domain $R$ is Mori if and only if the monoid $R^{\bullet}$ is Mori.

Note that  a domain $R$ is primary if and only if $R$ is one-dimensional and local (\cite[Proposition 2.10.7]{Ge-HK06a}).
Every  primary Mori monoid is strongly primary with $v$-$\spec (H) = \{\emptyset, H \setminus H^{\times}\}$ whence every one-dimensional local Mori domain is strongly primary (\cite[Proposition 2.10.7]{Ge-HK06a}. A survey on the ideal theory of Mori domains is given by  Barucci in \cite{Ba00}).

All finitely primary monoids (including all numerical monoids) are strongly primary (\cite[Section 2.7]{Ge-HK06a}).
Examples of finitely primary domains which are not Mori can be found in \cite[Sections 3 and 4]{HK-Ha-Ka04}.
Moreover, \cite[Examples 4.6 and 4.7]{HK-Ha-Ka04} are not  multiplicative monoids of domains; Example 4.6 is Mori,
while Example 4.7 is not.    Puiseux monoids (these are additive submonoids of $(\Q_{\ge 0}, +))$, which are strongly primary, are discussed in \cite{Ge-Go-Tr20} and strongly primary monoids stemming from module theory can be found in \cite[Theorems 5.1 and 5.3]{Ba-Ge-Gr-Sm15}.

An example going back to Grams (\cite[Example 1.1]{An-An-Za90} or \cite[Example 1.1.6]{Ge-HK06a}) exhibits an atomic one-dimensional local domain which is not a BF-domain whence it is neither strongly primary nor locally tame (see Definition \ref{3.1} below).

\section{On the arithmetic  of strongly primary monoids and domains} \label{3}

We start with the concept  of local tameness as given in \cite{Ge-HK06a}.

\begin{definition} \label{3.1}
	Let $H$ be an atomic monoid.
	\begin{enumerate}
		\item For an element $a \in H$, let $\omega (a)$ denote the smallest $N \in \mathbb N_0 \cup \{\infty\}$ with the following property: if $n \in \N$ and $a_1, \ldots, a_n \in H$ with $a \mid a_1 \cdot \ldots \cdot a_n$, then there is a subset $\Omega \subset [1,n]$ such that
		\[
		|\Omega| \le N \quad \text{and} \quad a \mid \prod_{\lambda \in \Omega} a_{\lambda} \,.
		\]
		We set $\omega (H) = \sup \{\omega (u) \mid u \in \mathcal A (H)\}$.
		
		\item For an element $u \in \mathcal A (H_{\red})$, let $\mathsf t (u)$ denote the smallest $N \in \mathbb N_0 \cup \{\infty\}$ with the following property: if $a \in H$ with $\mathsf Z (a) \cap u \mathsf Z (H) \ne \emptyset$ and $z \in \mathsf Z (a)$, then there is a $z' \in \mathsf Z (a) \cap u \mathsf Z (H)$ such that $\mathsf d (z,z') \le N$.
		
		\item $H$ is said to be
		\begin{enumerate}
			\item {\it locally tame} if $\mathsf t ( u) < \infty$ for all $u \in \mathcal A (H_{\red})$, and
			\item {\it $($globally$)$ tame} if $\mathsf t (H) = \sup \{ \mathsf t ( u)  \mid u \in \mathcal A (H_{\red}) \} < \infty$.
		\end{enumerate}
	\end{enumerate}
\end{definition}

If $u\in \mathcal A(H)$, we let $\mathsf t(u)=\mathsf t(uH^{\times})$. For a prime element $u \in H$, we have $\omega (u)=1$ and $\mathsf t ( u)=0$, thus $\omega(H)=1$ for a factorial monoid. If $u \in \mathcal A (H)$ is not prime, then $\omega (u) \le \mathsf t (u)$ whence for a non-factorial monoid we have $\omega (H) \le \mathsf t (H)$. Moreover, $H$ is globally tame if and only if $\omega(H)<\infty$ \cite[Proposition 3.5]{Ge-Ka10a}. Every Mori monoid satisfies $\omega (a) < \infty$ for all $a \in  H$ (\cite[Theorem 4.2]{Ge-Ha08a}, \cite[Proposition 3.3]{F-G-K-T17}) but this need not hold true for the local tame degrees $\mathsf t ( \cdot)$. If $H$ is an atomic monoid, since $\sup \mathsf L (a) \le \omega (a)$ for all $a \in H$ (\cite[Lemma 3.3]{Ge-Ha08a}), the finiteness of the $\omega (a)$ values (hence, in particular, local tameness) implies that $H$ is a \BF-monoid.

We continue with a  simple reformulation of local tameness which we use in the following. Clearly, for an atom $u \in \mathcal A (H_{\red})$, the local tame degree $\mathsf t (u)$ is the smallest $N \in \mathbb N_0 \cup \{\infty\}$ with the following property:
\begin{itemize}
	\item[] For every multiple $a \in H_{\red}$ of $u$ and any factorization $z=v_1 \cdot \ldots \cdot v_n$ of $a$ which does not contain $u$, there is a subproduct which is a multiple of $u$, say $v_1 \cdot \ldots \cdot v_m$, and a refactorization of this subproduct which contains $u$, say $v_1 \cdot \ldots \cdot v_m=uu_2 \cdot \ldots \cdot u_{\ell}$ such that $\max \{\ell, m\} \le N$.
\end{itemize}
\medskip

Recall that a monoid is {\it half-factorial} if all the factorizations of an element in $H$ are of the same length, equivalently, if
the set of distances $\Delta(H)$ is empty.

\begin{proposition}\label{lambdarho}
	\label{nonemptyint}
	Let $H$ be  a strongly primary monoid.
	\begin{enumerate}
	\item
For every atom $u\in H$, we have $\omega(u)\le \mathcal M(u)$. Thus $\omega(u)<\infty$.
	\item
		Assume that  $\Lambda(H)<\infty$. Then $\rho_{\Lambda(H)}(H)=\infty$, so $\rho_k(H)=\infty$ for all $k \ge \Lambda(H)$. In particular,  $H$ is not half-factorial.	
	\item
For every atom  $u\in H$ we have  $\mathsf t(u)\le \rho_{\omega(u)}(H)$. Hence, if  $\rho_{\omega(u)}(H)<\infty$ for every atom $u\in H$, then $H$ is locally tame.
	\end{enumerate}
\end{proposition}	

\begin{proof}
1. Since $u$ divides each product of $\mathcal M(u)$ non-units in $H$, we see that $N=\mathcal M(u)$ satisfies the property in Definition \ref{3.1}.1, and since $\omega(u)$ is the least $N$ satisfying this property, it follows that   $\omega(u)\le \mathcal M(u)$.		
	
2.
For every $n\in\mathbb N$, a product of $n$ atoms  is also a product of $k$ atoms for some positive integer $k \le\Lambda(H)$. Hence $\rho_{\Lambda(H)}(H)\ge \rho_k(H)\ge n$. Thus $\rho_{\Lambda(H)}(H)=\infty$.	

3. Let $a$ be an element of $H$ that is divisible by an atom $u$. A factorization of $a$ has a subproduct that is divisible by $u$ and of length $k\le \omega(u)$. Hence every factorization of this subproduct is of length $\le \rho_k(H)\le \rho_{\omega(u)}(H)$. By the reformulation of Definition \ref{3.1}.1 we obtain that $\mathsf t(u)\le \rho_{ \omega(u)}(H)$.
\end{proof}

\medskip

A monoid $H$ is said to be a
\begin{itemize}
	\item {\it valuation monoid} if for all $a, b \in H$ we have $a \mid b$ or $b \mid a$.
	
	\item {\it discrete valuation monoid}  if $H_{\red} \cong (\N_0, +)$.
\end{itemize}

\begin{lemma} \label{3.3}
	Let $H$ be a  monoid with $\mathfrak m = H \setminus  H^{\times} \ne \emptyset$.
	\begin{enumerate}
		\item $H$ is a primary valuation monoid if and only if $H$ is a completely integrally closed valuation monoid if and only if $H_{\red}$  is isomorphic to a monoid of non-negative elements of a non-zero subgroup of $(\R, +)$.
				
		\item The following conditions are equivalent:
\begin{enumerate}
\item
 $H$ is a discrete valuation monoid.

\item
$H$ is  atomic  and $\mathfrak m$ is principal.

\item
$H$ is a completely integrally closed  strongly primary monoid.
\end{enumerate}

\item\cite[Theorem 16.4 g)]{HK98}
A valuation monoid is discrete if and only if it is atomic if and only if it is strongly primary.
\end{enumerate}
\end{lemma}

\begin{proof}
	For 1. and for the equivalence of the first two items of 2. we refer to  \cite[Theorems 15.5 and 16.4]{HK98} and to \cite[Section 3]{Ge93b}. To complete the proof of 2., we note that $(a)	\Rightarrow (c)$, so it is enough to prove the implication $(c)\Rightarrow (b)$.
	Thus it remains to show that $\mathfrak m$ is principal. Let $u \in \mathcal A (H)$. Then $\mathfrak m^{\mathcal M (u)} \subseteq uH$. If $\mathcal M (u)=1$, then $\mathfrak m = uH$ and we are done. Assume that $\mathcal M (u) >1$. Then $\mathfrak m^{\mathcal M (u)} \subseteq u H$, whence $u^{-1} \mathfrak m^{\mathcal M (u)} \subseteq \mathfrak m$ and $u^{-1}\mathfrak m^{\mathcal M (u)-1} \subseteq (\mathfrak m \colon \mathfrak m)$. Since $ H$ is completely integrally closed, we have $(\mathfrak m \colon \mathfrak m) \subseteq H$ and thus $\mathfrak m^{\mathcal M (u)-1} \subseteq u H$,  contradicting the minimality of $\mathcal M (u)$.
\end{proof}

\begin{lemma}\label{wdhprimary}
Let $(H, \mathfrak m)$ be a primary monoid and $\mathfrak n = \widehat H \setminus \widehat H^{\times}$.
\begin{enumerate}
\item If  every element of $\mathfrak n$ has a power lying in $\mathfrak m$, then $\widehat H$ is primary.

\item If $(H\DP \widehat H)\ne\emptyset$,  then $\widehat H$ is primary if and only if  every element of $\mathfrak n$ has a power lying in $\mathfrak m$.
\end{enumerate} 	
\end{lemma} 	

\begin{proof}
1. Let $x,y\in \mathfrak n$.  There exists an $n\in \mathbb N$ such that $x^n=m\in \mathfrak m$. Since $H$ is primary, by Lemma \ref{prim}.1., there exists a $k \in \mathbb N$ such that $m^k \in yH$. Thus $x^{nk}\in yH$, implying that $\widehat H$ is primary.

2. Assume that $\widehat H$ is primary. Let $c\in \mathfrak m(H\DP \widehat H)$ and $x\in \mathfrak n$. As $\widehat H$ is primary, and $\mathfrak  m\subseteq \mathfrak n$ by  Lemma \ref{2.3}, we obtain that  for some integer $n\in \mathbb N$ we have $x^n\in c\widehat H \subseteq \mathfrak m$.
We  complete the proof using item 1. 	
\end{proof}

\begin{lemma} \label{3.4}
Let $(H, \mathfrak m)$ be a strongly primary monoid and $\mathfrak n = \widehat H \setminus \widehat H^{\times}$.
	\begin{enumerate}
		\item If $x \in \mathsf q (H)$, then $\Lambda (H \setminus x H) < \mathcal M (x)<\infty$.
		
		\item We have $\Lambda (H) < \infty$ if and only if there is a $c \in \mathfrak m$ with $\Lambda ( \{c^m \mid m \in \N_0 \}) < \infty$.
		
		\item If $\Lambda (H) =\infty$, then every element of $\mathfrak n$ has a power lying in $\mathfrak m$, so $\widehat H$ is primary.
	\end{enumerate}
\end{lemma}

\begin{proof}
1. If $a \in H \setminus xH$, then $a \notin \mathfrak m^{\mathcal M (x)} \subseteq xH$. Thus $\Lambda (H \setminus x H) < \mathcal M (x)$.
	
2. Let $c \in \mathfrak m$ such that $\Lambda ( \{c^m \mid m \in \N_0 \}) < \infty$. Let $a \in H$. By  \cite[Lemma 2.7.7.1]{Ge-HK06a}, we have   $a\mid c^k$ for some $k \in \mathbb N$, so there is an $n\le k$ in $\N_0$ such that $a=c^nb$, where $b \in H$ is not divisible by $c$. Now 1. implies that
	\[
	\Lambda (a) \le \Lambda (c^n) + \Lambda (b) \le \Lambda ( \{c^m \mid m \in \N_0 \}) + \mathcal M (c) \,.
	\]
	Thus $\Lambda (H) \le \Lambda ( \{c^m \mid m \in \N_0 \}) + \mathcal M (c)$, and the reverse implication is trivial.
	
3. Suppose there is an $x \in \mathfrak n$ such that no power of $x$ belongs to $\mathfrak m$, and we will prove that $\Lambda(H)<\infty$. Let $d \in \mathfrak m$ such that $dx^n \in H$ for all $n \in \N$. We choose an element $c\in \mathfrak m$ and assert that $\Lambda(c)<\mathcal M(d)+\mathcal M(x)$, implying that
	$\Lambda(H)<\mathcal M(d)+\mathcal M(x)$. If $c \notin dH$, then $\Lambda (c) < \mathcal M (d)$ by 1.
	Suppose that $c\in dH$. Then there exists an integer $k\in\mathbb N$ such that $c\in (dx^kH\setminus dx^{k+1}H)$, because otherwise we would have $c(1/x)^n \in H$ for all $n \in \N$ implying that $1/x\in \widehat H$. Thus  $c=(dx^k)b$, where $b\in H\setminus xH$ whence $\Lambda(b)<\mathcal M(x)$ by 1. If $dx^k \in dH$,  then  $x^k\in H\setminus \mathfrak m$, so $x$ is invertible in $H$, a contradiction.  This implies that  $dx^k\in H\setminus  dH$ whence $\Lambda(dx^k)<\mathcal M(d)$ by 1.  Thus $\Lambda (c) = \Lambda(dx^kb)\le\Lambda(dx^k)+\Lambda(b)<\mathcal M(d)+\mathcal M(x)$. By Lemma \ref{wdhprimary}, $\widehat H$ is primary.
\end{proof}

  The converse of Lemma \ref{3.4}.3 is false even for domains (see Example \ref{Nagata} below and Lemma \ref{wdhprimary}).
Let $(H,\mathfrak m)$ be a strongly primary monoid. If $\widehat H$ is not primary (equivalently, if $\mathfrak m$ is not the unique prime s-ideal of $\widehat H$), then $\Lambda(H)<\infty$. In particular, if $R$ is a strongly primary domain such that $\widehat R$ is not local, then $\Lambda(R)<\infty$.

\begin{lemma} \label{3.5}
Let $(H, \mathfrak m)$ be a strongly primary monoid such that $\Lambda (H)=\infty$. Let $\widetilde{\mathfrak m} = \widetilde H \setminus {\widetilde H}^{\times}$ and $\mathfrak n = \widehat H \setminus {\widehat H}^{\times}$. Then
	\begin{enumerate}
		\item  $\widetilde{\mathfrak m} = \mathfrak n$.
		
		\item For every $x \in \mathfrak n$, we have $x^n \in \mathfrak m$ for all sufficiently large $n \in \N$.
		
		\item If $(H \colon \widehat H) \ne \emptyset$, then $\widehat H$ is a primary  valuation monoid.
	\end{enumerate}
\end{lemma}

\begin{proof}	
	1. By \ref{eq2.1}, we have $\widetilde H \subseteq  \widehat H$, so by  Lemma \ref{2.3}, we infer that $\widetilde{\mathfrak m} \subseteq \mathfrak n$.
	Since $\Lambda (H)=\infty$, Lemma \ref{3.4}.3 implies that every element of $\mathfrak n$ has a power lying in $\mathfrak m$, whence $\mathfrak n \subseteq \widetilde H \setminus {\widetilde H}^{\times} = \widetilde{\mathfrak m}$.

	2. Let $x \in \mathfrak n$.  Since $\Lambda (H)=\infty$, Lemma \ref{3.4}.3 implies that there is a $k \in \N$ such that $x^k \in \mathfrak m$. Since $H$ is primary, there is a $q_0 \in \N$ such that $x^{q_0k+r}= (x^k)^{q_0}x^r \in \mathfrak m$ for all $r \in [0, k-1]$. If $n \in \N$ with $n \ge q_0k$, then $n=qk+r$, where $q \ge q_0$ and $r \in [0,k-1]$, and $x^n = x^{k(q-q_0)} x^{q_0k+r} \in \mathfrak m$.
	
	3. By 	Lemma \ref{wdhprimary}, $\widehat H$ is primary.
	
	Assume to the contrary, that  $\widehat H$ is not a valuation  monoid. Thus there exists an element $x\in \mathsf q (H)$ such that $x, x^{-1}\notin \widehat H$. If $s\in (H\DP \widehat H)$, then $s^{n}x, s^{n}x^{{-1}}\in H$ for  all sufficiently large $n \in \N$. Hence there exists $c\in (H\DP \widehat H)$ such that $cx, cx^{-1}\in H$.  For each $k \in \N$, let $n_k\in \N$  be the smallest integer such that $c^{n_k}x^k\in H$, and let $\widehat n_k\in \N$  be the smallest integer such that $c^{\widehat n_k}x^k\in \widehat H$.   Thus, by definition,  $1\le\widehat n_k \le n_k\le k$, $\widehat n_{k+1} \le \widehat n_k + 1$. Also $c^{n_k}x^k$ is not divisible by $c$ in $H$: for $n_k=1$ this holds since $x^k\notin H$, and for $n_k>1$ this follows from the minimality of $n_k$. Since $c\widehat H\subseteq H$, we obtain that $n_{k}\le \widehat  n_k+1$, whence $\widehat n_k \le n_k\le \widehat n_k+1$. As $\widehat   H$ is root-closed by  Lemma \ref{2.2}.1, we obtain by Lemma \ref{2.1}.1 (applied to $\widehat H$) that the sequence $\widehat n_k$ is increasing. Thus we infer that $|n_{k+1}-n_k|\le |\widehat n_{k+2}-\widehat n_k|\le 2$.
	Since $(H\DP \widehat H)\ne \emptyset$, Lemma \ref{2.2} (items 2 and 3) implies that  $\widehat H$ is completely integrally closed and that, for every $m \in \N$, $c^m x^k\in \widehat H$ for just finitely many  $k$'s. This implies that    $\lim_{k\to\infty}\widehat n_k=\infty$ whence $\lim_{k\to\infty}n_k=\infty$.

	Proceeding in the same way with the element $x^{-1}$ as with the element $x$, we obtain   a sequence $n'_k$ having all the properties of the sequence $(n_k)_{k \ge 1}$ with respect to the element $x^{-1}$. Then for all $k\in \N$, the element
	\[
	c^{n_k+n'_k}=(c^{n_k} x^k)(c^{n'_k} x^{-k})
	\]
	is a product of two elements not divisible by $c$ in $H$.
	
	Let $n \in \N$ and  let $k \in \N$ be  maximal  such that $n_k+n'_k\le n$. Then $n_k+n'_k \le n \le n_{k+1}+n'_{k+1}\le n_k+n'_k +4 $. This implies that
	\[
	c^n = c^{n_k+n'_k}c^f =  (c^{n_k} x^k)(c^{n'_k} x^{-k})c^f  \quad \text{for some } \quad f \in [0,4] ,
	\]
	whence $\Lambda (c^n) \le \Lambda (c)+\Lambda (c) + \Lambda (c^f)$. Thus Lemma \ref{3.4}.2 implies that  $\Lambda(H)<\infty$, a contradiction.
\end{proof}

\begin{theorem} \label{zerocond}
Let $(R,M)$ be a strongly primary domain.  If 	$(R \DP \widehat R)=(0)$, then $\Lambda(R) < \infty$.
\end{theorem}

\begin{proof}
Assume to the contrary that $\Lambda(R)=\infty$. Let $\mathfrak  n=\widehat R\setminus (\widehat R)^{\times}$.	Choose a nonzero element $c\in M$. Let $n$ be a positive integer. Since $(R \DP \widehat R)=(0)$, we infer that $(c^{n-1}\mathfrak n)\widehat R \not\subseteq R$, whence 	there exists an element  $x\in \mathfrak n$ such that $c^{n-1} x\notin R$.

Since $\Lambda(R)=\infty$,  Lemma \ref{3.5}.2 implies that $x^i\in M$ (equivalently, $x^i\in R$) for all sufficiently large integers $i\in \mathbb N$.	Let $i \in \N$ be maximal  such that $c^{n-1}x^i \notin R$. Since  $R$ is primary,  there exists a maximal nonnegative integer $k$ such that $c^{n-1}c^kx^i\notin M$. Set $y=c^kx^i$, so $c^{n-1}y\notin R$. We have $c^ny^j\in M$ for all $j\in \mathbb N$ and $c^{n-1}y^j\in M$ for all $j>1$ since $c^{n-1}x^j\in \mathfrak m$ for all $j>i$.
We have $y^e\in M$ for some positive integer $e$.
Hence
$$
(1-y)(1+y+\cdots +y^{e-1})=1-y^e\in R^{\times}.
$$
Thus
$$
c^n(1-y)\in R, \quad \text{ and} \quad \frac {c^{n}} {1-y}=c^n\frac {1+y+\cdots +y^{e-1}} {1-y^{e}}\in R.
$$
We see that   $c^n(1-y)$ and $\frac {c^{n}} {1-y}$ are not divisible by $c$  in $R$.
Thus $c^{2n}$ is a product of two elements that are not divisible by $c$ in $R$:
$$
c^{2n}=\Big(c^{n}(1-y)\Big) \, \Big(\frac {c^{n}}{1-y}\Big).
$$
Hence $\Lambda (c^{2n}) < \mathcal M (c)+\mathcal M (c)$.
By  Lemma  \ref{3.4}.2 applied to $c^2$ replacing $c$, we conclude that $\Lambda(R)<\infty$, a contradiction.
\end{proof} 	

Theorem \ref{zerocond} is false for monoids by Example \ref{3.13}.1 below.

\begin{theorem} \label{main}
	\begin{enumerate}
		\
		
		\item [(a)]
		Let $(H, \mathfrak m)$ be a strongly primary monoid, and let 	$\mathfrak f=(H\DP \widehat H)$. Then the first two conditions below are equivalent,  and each of the first four conditions implies its successor.	Moreover, if $\mathfrak f	\ne \emptyset$, then all nine conditions are equivalent.		
\begin{enumerate}
			\item[(1)] $H$ is globally tame.
			
			\item[(2)] $\bigcap_{u\in \mathcal A(H)} uH\ne \emptyset$.

			\item[(3)] $\rho (H) < \infty$.
			
			\item[(4)] $\rho_k(H)<\infty$ for all $k \in \N$.
			
			\item[(5)] $\Lambda (H)=\infty$.
			
			\item[(6)] $\widehat H$ is a primary  valuation monoid.
			
			\item[(7)] 	$\widehat H$ is a valuation monoid.				
			
			\item[(8)] 	$\mathfrak f\mathfrak m\subseteq \bigcap_{u\in \mathcal A (H)}u\widehat H$.
			
			\item[(9)]  $\mathfrak f\mathfrak m^2\subseteq \bigcap_{u\in \mathcal A (H)} uH$.		
	\end{enumerate}

		\item[(b)] Let $(R,M)$ be a strongly primary domain, and $\mathfrak f=(R\DP \widehat R)$. All the following nine conditions are equivalent:
		\begin{enumerate}
		 \item[(1)] $R$ is globally tame.
		
		 \item[(2)] $\bigcap_{u\in \mathcal A(H)} uH\ne \emptyset$.
			
		 \item[(3)] $\rho (R) < \infty$.
		
		 \item[(4)] $\rho_k(R)<\infty$ for all $k \in \N$.
		
		 \item[(5)] $\Lambda (R)=\infty$.
		
		 \item[(6)] $\widehat R$ is a primary  valuation domain and $\mathfrak f\ne(0)$.
		
		 \item[(7)] 	$\widehat R$ is a valuation domain and $\mathfrak f\ne(0)$.				
		
		 \item[(8)] 	$(0)\ne\mathfrak f\mathfrak m\subseteq \bigcap_{u\in \mathcal A (R)}u\widehat R$.
		
		 \item[(9)]  $(0)\ne\mathfrak f\mathfrak m^2\subseteq \bigcap_{u\in \mathcal A (R)} uR$.	
		\end{enumerate}
	\end{enumerate}
\end{theorem}	

\begin{proof}
	\begin{enumerate}
		\item[(a)]

		\
		
		$(1)\Rightarrow (2)$
		
		By assumption $\omega(H)<\infty$.  Let $c \in \mathfrak m$. Then for every $a \in \mathcal A (H)$ there is an $n_a \in \N$ such that $a \mid c^{n_a}$ whence $a \mid c^{\omega (H)}$. This implies that $c^{\omega (H)}\in \bigcap_{u\in \mathcal A(H)} uH$.

		$(2)\Rightarrow (1)$ 	Since $H$ is strongly primary, we have $\mathfrak m^k\subseteq \bigcap_{u \in \mathcal A (H)} uH$ for some positive integer $k$. Since for every atom $u$, we have $\omega(u)\le k$, it follows that $\omega(H)<\infty$, so that $H$ is globally tame.
		
		Thus the first two conditions are equivalent.

		$(1)\Rightarrow (3)$ See \cite[Theorem 1.6.6]{Ge-HK06a}.

		$(3)\Rightarrow (4)$ See Equation \eqref{elasticities}.
		
	   $(4)\Rightarrow (5)$ This follows from Proposition \ref{lambdarho}.2.
				
		\medskip
		
			Now assume that $\mathfrak f\ne(0)$.	
		
		$(5)\Rightarrow (6)$ See Lemma \ref{3.5}.3.
	
	$(6)\Rightarrow (7)$ Obvious.
	
	$(7)\Rightarrow (8)$ 	Let $u$ be an atom in $H$.  Since $u\notin \mathfrak f\mathfrak m$ and $\mathfrak f\mathfrak m$ is an $s$-ideal of the valuation monoid $\widehat H$, it follows that $\mathfrak f\mathfrak m\subseteq u\widehat H$.  Thus the assertion follows.
	
	$(8)\Rightarrow (9)$
		
	Let $u$ be an atom in $H$.   Since $ \mathfrak f\mathfrak m\subseteq u\widehat H$, we infer that  $\mathfrak f^2 \mathfrak m\subseteq u \mathfrak f \widehat H \subseteq uH$. The assertion follows.
	
	$(9)\Rightarrow (2)$ Obvious.
			
		\item[(b)]
		
		By Theorem \ref{zerocond}, condition (5) implies that $\mathfrak f\ne(0)$. Thus each of the first five conditions implies  that $\mathfrak f\ne(0)$. Obviously, each of the other four conditions  implies that $\mathfrak f\ne(0)$.
		By item a., all the nine conditions are equivalent.
		\end{enumerate} 	
\end{proof}

\begin{theorem}\label{tame}
	
	\
	\begin{enumerate}
		\item
		Let $H$ be a strongly primary monoid.
		\begin{enumerate}

		\item		
If $\Lambda(H)<\infty$, then $H$ is locally tame, but not globally tame.

\item
If	$\Lambda(H)=\infty$ and $( H\DP \widehat H)\ne \emptyset$, then $H$ is globally tame.

\item
$H$ is locally tame if  either  $\Lambda(H)<\infty$, or $(H\DP \widehat H)\ne \emptyset$.
\end{enumerate}

		\item
		Let $R$ be  a strongly primary domain.
		\begin{enumerate}
			\item
			$R$ is locally tame.
			
			\item
			$R$ is globally tame if and only if $\Lambda(R)=\infty$.
		\end{enumerate}
	\end{enumerate} 	
	\end{theorem}

\begin{proof}

\	
\begin{enumerate}
\item

\
\begin{enumerate}
\item
   Let $u \in \mathcal A (H_{\red})$,  $a$ be a multiple of $u$,  and let $a=v_1\dots v_n$ be a product of atoms.  There exists a subproduct of $v_1\dots v_n$  of length $\le \omega(u)$ that is divisible by $u$. This subproduct has a refactorization of the form $ub$ where $b$ is a product of at most $\Lambda(H)$ atoms. Hence $\mathsf t(u)\le \max \{\omega(u),\Lambda(H)+1 \} <\infty$, so $H$ is locally tame.
 \smallskip

The monoid $H$ is not globally tame by the implication $(1)\Rightarrow (5)$ of Theorem \ref{main}.

\item
	If	 $(\widehat H\DP H)\ne\emptyset$, then all nine conditions  of Theorem \ref{main} (a) are equivalent. In  particular, if $\Lambda(H)=\infty$    (condition (5)), then $H$ is globally tame (condition (1)).
	
\item
This follows immediately from the previous two items.
\end{enumerate} 	
 \item
By Theorem \ref{zerocond}, if $(R\DP\widehat R)=(0)$, then $\Lambda(R)<\infty$. Hence  item 2. (for domains) follows from item 1.(for monoids).
\end{enumerate} 	
\end{proof}

In the next corollary we answer in the positive Problem 38 in \cite{C-F-F-G14}.

\begin{corollary}
	A one-dimensional local Mori domain $R$ is locally tame. Moreover, $R$ is globally tame if and only if $\Lambda(R)=\infty$.
\end{corollary} 	

\begin{proof}
A one-dimensional local Mori domain is strongly primary. Thus the corollary follows from Theorem \ref{tame} 2.	
\end{proof} 	

In the next proposition we deal with two significant special cases of strongly primary monoids.

\begin{proposition} \label{3.7}
	Let $(H, \mathfrak m)$ be a  strongly primary monoid.
	\begin{enumerate}
		\item Let $H$ be seminormal. Then $(H \colon \widehat H) \supseteq \mathfrak  m$, so $H$ is locally tame,  and all conditions of Theorem  \ref{main} (a) are equivalent. If $\widehat H$ is Krull, then $H$ is Mori.
		
		\item Let $H$ be Mori with $(H \colon \widehat H) \ne \emptyset$. Then $\widehat H$ is Krull,  $\widehat H_{\red}$ is finitely generated, and all conditions of Theorem \ref{main} (a) are equivalent to $\widehat H$ being a discrete valuation monoid.
	\end{enumerate}
\end{proposition}

\begin{proof}
	1. We have  $(H \colon \widehat H) \supseteq \mathfrak m$ by Lemma \ref{2.4}.2. Thus $H$ is locally tame (Theorem \ref{tame}), and the  conditions of Theorem \ref{main} (a) are equivalent. If $\widehat H$ is Krull, then $H$ is Mori by \cite[Lemma 2.6]{Re13a}.
	
	2. Suppose that $H$ is Mori and $(H \colon \widehat H) \ne \emptyset$. Then $\widehat H$ is Krull and $\widehat H_{\red}$ is finitely generated by \cite[Theorem 2.7.9]{Ge-HK06a}. Since  $(H \colon \widehat H) \ne \emptyset$, all conditions of Theorem \ref{main} (a) are equivalent to $\widehat H$ being a valuation monoid (condition (7)). Since Krull monoids are  atomic,  Lemma \ref{3.3}.3 implies that Krull valuation monoids are discrete.
\end{proof}

\begin{proposition}\label{nodom}
Let $H$ be a strongly primary monoid that satisfies one of the following two properties:
\begin{enumerate}
\item
$H$ is not locally tame.

\item
$\Lambda(H)=\infty$ and $(H\DP \widehat H)=\emptyset$.
\end{enumerate} 	
Then $H$ is not the multiplicative monoid of a domain.
\end{proposition} 	

\begin{proof}
For (1) see Theorem \ref{tame} (b). 	For (2) see Theorem \ref{main} (b).
\end{proof}

For a strongly primary Mori monoid, that  satisfies both conditions of Proposition \ref{nodom}, see Example \ref{3.13}.1 below or \cite[Proposition 3.12]{Ge-Go-Tr20}.

\begin{remark} \label{3.10} Let $R$ be a strongly primary domain. Then the Theorems \ref{zerocond} and  \ref{main} show that either ($\Lambda (R)<\infty$ and $\rho (R)=\infty$) or ($\rho (R)< \infty$ and $\Lambda (R)=\infty$). This was proved for one-dimensional local noetherian domains in \cite[Corollary 3.7]{Ha02a} and it was assumed as an additional abstract property for strongly primary monoids in the study of weakly Krull domains in \cite[Corollary 4.11]{Ha04c}.
\end{remark}
	
We present some examples related to the complete integral closure of a strongly primary domain. Let $R$ be a strongly primary domain such that $(R\DP \widehat R)\ne(0)$.  By  Theorem \ref{main} (b),  if $\widehat R$ is a valuation domain, then $\widehat R$ is primary. The converse  is false as shown in Example \ref{Nagata} below. Moreover,  $\widehat R$ is strongly primary if and only if  $\widehat R$ is a discrete valuation domain by Lemma \ref{3.3}.3. In Example \ref{valnotdis},   $\widehat R$ is a valuation domain, but it is not  discrete.
On the positive side, $\widehat R$ is a discrete valuation domain if $R$ is Mori by Proposition \ref{3.7}.2.
Clearly, the domain  $\widehat R$ need not be primary.  In particular, if $R$ is a  one-dimensional local noetherian domain, then $\widehat R$, which is equal to the integral closure of $R$, is not necessarily a local ring.

For Example \ref{Nagata} below, we need Proposition \ref{existspm}.

\begin{proposition}\label{existspm}~

	\begin{enumerate}
		\item		
		Let $T$ be a primary monoid. There exists a strongly primary submonoid $(H,\mathfrak m)$ of $T$ such that $T=\widehat H$ and $(H\DP T)\ne\emptyset$ if and only if $T$ is  completely integrally closed. Moreover, in this case, the monoid $(H, \mathfrak m)$ can be chosen such that $\mathfrak m$ is a principal $s$-ideal of $T$, whence  $T=(\mathfrak m \DP \mathfrak m)$ and $\mathfrak m\subseteq (H\DP T)$. The  conditions of Theorem \ref{main} (a) are satisfied if and only if $T$  is a  discrete valuation monoid, and just in this case we may choose $H=T$.
		
		\item		
		Let $T$ be a primary  domain. There exists a strongly primary subring $(R,M)$ of $T$ such that $T=\widehat R$ and $(R\DP T)\ne(0)$ if and only if $T$ is  completely integrally closed. Moreover, in this case, the ring $(R, M)$ can be chosen such that $T=\widehat R$ and the ideal $MT$  of $T$ is generated by two elements.   If furthermore,  $T$ contains a field, then we may choose $R$ such  that  $M$ is a principal ideal of $T$, whence $T=(M\DP M)$ and $M\subseteq (R\DP T)$. The  conditions of Theorem \ref{main} (b) are satisfied if and only if $T$  is a  discrete valuation domain, and just in this case we may choose $R=T$.
		\end{enumerate}	
\end{proposition} 	

\begin{proof}
	By Lemma \ref{2.2}.2, if $H$ is a monoid such that $(H\DP \widehat H)\ne\emptyset$, then $\widehat H$ is completely integrally closed. Thus, in both statements, we just have to  prove the converse. We start with the proof of the second statement.
	
2.(i)
		Let $\mathfrak n=T\setminus T^{\times}$, and let $c$ be an element of $\mathfrak n$. Let  $H=cT\cup \{1\}$, and let $\mathfrak m=cT$. Thus $(H,\mathfrak m)$ is a submonoid of $(T,\mathfrak n)$.  Since $T$ is primary, if $x\in H$, then, $c^n\in xT$ for some integer $n\in \mathbb  N$. Thus $\mathfrak m^{n+1}c^{n+1}T\subseteq  xcT=c\mathfrak m$. It follows that $H$ is a strongly primary monoid. Also $(H:T)\supseteq \mathfrak m$. Hence $T\subseteq \widehat H$. Since $T$ is completely integrally closed, we obtain that $T=\widehat H$. For the last sentence see Lemma \ref{3.3}.3.

2.(ii)
		The domain $T$ is local since it is primary. Let $N$ be the maximal ideal of $T$. Let $c$ be a non-zero element of $N$, and let  $F$ be the prime field contained in the quotient field of $T$. Set
		$A=(F+cT)\cap T=(F\cap T)+cT$, thus $A$ is a subring of $T$. Let $P=A\cap N=(F\cap N)+cT$, so $P$ is a prime ideal of $A$. Set $R=A_P$, and $M=PA_P$. Thus $(R,M)$ is a local subring of $(T,N)$.
		
		If  $T$ contains a field (e.g., if $T$ has finite characteristic), then  $P=cT=PA_P=M$, thus $M$ is a principal ideal of $T$. Otherwise, $T$ has zero characteristic, and we may identify $F=\mathbb Q$. Thus $F\cap T$ is a local subring of $\mathbb Q$, that is, a localization of $\mathbb Z$ at a nonzero prime ideal. Hence $F\cap T$ is a discrete valuation domain. Let $d$ be a generator of the maximal ideal of $F\cap T$. Thus $MT=cT+dT$.
		
		  Let $x$ be a non-zero element of $M$, Since $T$ is primary,  and $MT$ is a finitely generated ideal of $T$, we have $(MT)^k\subseteq cxT\subseteq xR$ for some positive integer $k$. Hene $M^k\subseteq   xR$, so $(R,M)$ is strongly primary. Since $(0)\ne cT\subseteq (R\DP \widehat R)$, we infer that $T\subseteq \widehat R$. Since $T$ is completely integrally closed, we obtain that $T=\widehat R$.

For the last sentence of item 2. see Lemma \ref{3.3}.3.	

1. 	The first statement follows from the second one. Indeed,  we may use the monoid $H=R^{\bullet}$, so $\widehat H=T^{\bullet}$, where $R$ and $T$ are the domains in 2.
\end{proof}

\begin{example}\label{Nagata}~

\begin{enumerate}
\item
 	There is a  strongly primary  monoid $(H,\mathfrak m)$  such that    $\widehat H$ is primary completely integrally closed,  but  not a valuation monoid. Moreover, $\mathfrak m$ is a principal $s$-ideal of $\widehat H$, so   $\widehat H=(\mathfrak m \DP \mathfrak m)$, abd $\mathfrak m\subseteq (H\DP \widehat H)$. Thus none of the conditions of Theorem \ref{main} (a) holds, in particular $\Lambda(H)<\infty$.

\item
For any field $k$, there is a strongly primary domain $(R,M)$ containing $k$ such that   $\widehat R$ is a primary completely integrally closed domain,  but  not a valuation domain. Moreover, $M$ is a principal ideal of $\widehat R$, so $\widehat R=(M\DP M)$, and  $M\subseteq (R:\widehat R)$.  Thus none of the conditions of Theorem \ref{main} (b) holds, in particular $\Lambda(R)<\infty$.
\end{enumerate}
\end{example} 	

\begin{proof}
 First we prove the existence of a domain as in item 2.
 There exists a primary completely integrally closed domain $T$ containing $k$ that is not a valuation domain (we refer to \cite{Ri56b}, \cite{Na52a} and \cite{Na55a}, and briefly sketch the idea. Indeed, the field $K$ of Puiseux series over the algebraic closure $\overline k$ of $k$ is algebraically closed and it  has a discrete valuation that vanishes on $\overline k$  with value group $\mathbb  Q$. As follows from the cited papers, this fact implies the existence of a primary completely integrally closed domain containing $\overline k$). By Proposition \ref{existspm}.2.  there exists a subring $R$ of $T$ that is strongly primary and such that $T=\widehat R$ and $M$ is a principal ideal of $T$. Since $\widehat R$ is not a valuation domain, by Theorem \ref{main} (b), none of the nine conditions of this theorem are satisfied, in particular, $\Lambda(R)<\infty$.

As for item 1., we define $H=R^{\bullet}$.  so $\widehat H=T^{\bullet}$, where $R$ is the domain in item 2.
\end{proof}
	
\begin{example}\label{valnotdis}
There is a strongly primary  domain $(R,M)$ such that  $\widehat R$ is a primary valuation domain, but not  strongly primary and $(R\DP \widehat R)=M$ is a principal ideal of $\widehat R$. Thus all the conditions of Theorem \ref{main} (b) are satisfied, in particular, $\Lambda(R)=\infty$.
\end{example} 	
	
\begin{proof}
	Let $F$ be a field, $A=F[X^{q} \mid q \text{ rational}, q\ge1]$, and let $P$ be the  maximal ideal of $A$ generated by the set $\{X^{q} \mid q \text{ rational }, q\ge1\}$. We set $R=A_P$ and $M=PR_P$. Clearly, each nonzero element  $r\in R$ is of the form $r=vX^q$, where $q\ge1$ is rational and $v$ is a unit in $R$, whence $M^{\lceil q\rceil+1}\subseteq rR$. Thus $R$ is strongly primary. It is easy to show  that $\widehat R$ is equal to  $B_Q$, where $B=F[X^{q} \mid q \text{ rational }, q>0]$ and  $Q$ is the  maximal ideal of $B$ generated by the set $\{X^{q} \mid q \text{ rational }, q>0\}$. Each nonzero  element of $\widehat  R$ is of the form $uX^q$, where $q>0$ is rational and $u$ is a unit in $\widehat R$. Clearly,  $\widehat R$ is a valuation domain and $(R\DP \widehat R)=M=X\widehat R$.
\end{proof}

\medskip
As mentioned in Section \ref{1}, Krull monoids with finite class group and C-monoids are locally tame. Furthermore, finitely generated monoids are locally tame (\cite[Theorem 3.1.4]{Ge-HK06a}) and hence the same is true for Cohen-Kaplansky domains and their monoids of invertible ideals (\cite[Theorem 4.3]{An-Mo92}). On the other hand,  examples of monoids or domains, that are not locally tame,  are rare in the literature. Thus we end this section with a brief overview.

\begin{example} \label{3.13}~

1. In contrast to Theorem \ref{zerocond} and Theorem \ref{main} (b), there is a  strongly primary monoid $H$ with the following properties (see \cite[Proposition 3.7 and Example 3.8]{Ge-Ha-Le07}):
	\begin{enumerate}
		\item[(i)] $H$ is Mori with $(H \DP \widehat H) = \emptyset$ and $\widehat H$ is a discrete valuation monoid,
		
		\item[(ii)]  $\rho (H)=\Lambda (H) = \mathsf c (H)= \infty$ and $H$ is not locally tame.
\end{enumerate}
	Moreover, $H$ is a submonoid of a one-dimensional local noetherian domain, although $H$ is not the multiplicative monoid of a domain since $H$ is not locally tame.
	
	\smallskip
	2. For every $\alpha \in \mathbb R \setminus \mathbb Q$, the additive monoid $H_{\alpha} = \{ (x,y) \in \N^2 \mid y < \alpha x \} \cup \{(0,0)\} \subset (\N_0^2, +)$ is a root-closed primary BF-monoid which is neither Mori, nor strongly primary, nor locally tame (\cite[Example 4.7]{Ge-Ha08a}).
	
	\smallskip
	3. The additive monoid $H = \{(a,b,c) \in \N_0^3 \mid a > 0 \ \text{or} \ b=c \} \subseteq (\N_0^3, +)$ is Mori with catenary degree $\mathsf c (H) = 3$ (whence, by \eqref{catenary}, all sets of lengths are arithmetical progressions with difference $1$) but it is not locally tame (\cite[Example 1]{HK08b}).
	
\smallskip
4. The domain $R = \Q[X^2, X^3]$ is one-dimensional noetherian, $\overline R = \Q[X]$ is factorial, and $(R \DP \overline R) = (X^2)$. Nevertheless, $R$ fails to be locally tame (\cite[Example 6.11]{Ge-Ha08b}).

	\smallskip
	5. Let $H$ be a Krull monoid with infinite cyclic class group. Then $H$ is locally tame if and only if its catenary degree is finite if and only if its set of distances is finite (\cite[Theorem 4.2]{Ge-Gr-Sc-Sc10}).
	
	\smallskip
	6. Let $H$ be a Krull monoid with class group $G$ such that every class contains a prime divisor. Then $H$ is locally tame if and only if $H$ is globally tame if and only if $G$ is finite (\cite[Theorem 4.4]{Ge-Ha08a}).
\end{example}

\providecommand{\bysame}{\leavevmode\hbox to3em{\hrulefill}\thinspace}
\providecommand{\MR}{\relax\ifhmode\unskip\space\fi MR }
\providecommand{\MRhref}[2]{%
  \href{http://www.ams.org/mathscinet-getitem?mr=#1}{#2}
}
\providecommand{\href}[2]{#2}

\end{document}